\documentclass[12pt]{amsart}
\usepackage{amsmath,amsthm,amssymb,cite}
\usepackage{graphicx}
\usepackage{subfig}

\usepackage{color}

\makeatletter
\@namedef{subjclassname@2010}{%
  \textup{2010} Mathematics Subject Classification}
\makeatother

\DeclareMathOperator{\RE}{Re} \DeclareMathOperator{\IM}{Im}

\numberwithin{equation}{section}
\newtheorem{theorem}{Theorem}[section]
\newtheorem{lemma}[theorem]{Lemma}

\theoremstyle{remark}
\newtheorem{remark}[theorem]{Remark}
\newtheorem{example}[theorem]{Example}

\begin{document}

\title{Starlikeness, convexity and close-to-convexity of harmonic mappings}

\thanks{The research work of the first author is supported by research fellowship from Council of Scientific and
Industrial Research (CSIR), New Delhi.}

\author[S. Nagpal]{Sumit Nagpal}
\address{Department of Mathematics, University of Delhi,
Delhi--110 007, India}
\email{sumitnagpal.du@gmail.com }

\author[V. Ravichandran]{V. Ravichandran}

\address{Department of Mathematics, University of Delhi,
Delhi--110 007, India}
\email{vravi68@gmail.com}

\begin{abstract}
In 1984, Clunie and Sheil-Small proved that a sense-preserving harmonic function whose analytic part is convex, is univalent and close-to-convex. In this paper, certain cases are discussed under which the conclusion of this result can be strengthened and extended to fully starlike and fully convex harmonic mappings. In addition, we investgate the properties of functions in the class $\mathcal{M}(\alpha)$ $(|\alpha|\leq 1)$ consisting of harmonic functions $f=h+\overline{g}$ with $g'(z)=\alpha zh'(z)$, $\RE (1+{zh''(z)}/{h'(z)} )>-{1}/{2} $ $ \mbox{ for } |z|<1  $. The coefficient estimates, growth results, area theorem and bounds for the radius of starlikeness and convexity of the class $\mathcal{M}(\alpha)$ are determined. In particular, the bound for the radius of convexity is sharp for the class $\mathcal{M}(1)$.
\end{abstract}



 \maketitle

\section{Introduction}
Let $\mathcal{H}$ denote the class of all complex-valued harmonic functions $f$ in the unit disk $\mathbb{D}=\{z \in \mathbb{C}:|z|<1\}$ normalized by $f(0)=0=f_{z}(0)-1$. Let $\mathcal{S}_{H}$ be the subclass of $\mathcal{H}$ consisting of univalent and sense-preserving functions. Such functions can be written in the form $f=h+\bar{g}$, where
\begin{equation}\label{eq1.1}
h(z)=z+\sum_{n=2}^{\infty}a_{n}z^{n}\quad\mbox{and}\quad g(z)=\sum_{n=1}^{\infty}b_{n}z^n
\end{equation}
are analytic and $|g'(z)|<|h'(z)| $ in $\mathbb{D}$. Let  $\mathcal{S}_{H}^{0}:=\{f \in \mathcal{S}_{H}:f_{\bar{z}}(0)=0\}$. Observe that $\mathcal{S}_{H}$ reduces to $\mathcal{S}$, the class of normalized univalent analytic functions, if the co-analytic part of $f$ is zero. In $1984$, Clunie and Sheil-Small \cite{cluniesheilsmall} investigated the class $\mathcal{S}_{H}$ as well as its geometric subclasses. Let $\mathcal{S}_{H}^{*}$, $\mathcal{K}_{H}$ and $\mathcal{C}_{H}$ (resp.\  $\mathcal{S}^*$, $\mathcal{K}$ and $\mathcal{C}$) be the subclasses of $\mathcal{S}_{H}$ (resp.\ $\mathcal{S}$) mapping $\mathbb{D}$ onto starlike, convex and close-to-convex domains, respectively. Denote by $\mathcal{S}_{H}^{*0}$, $\mathcal{K}_{H}^{0}$ and $\mathcal{C}_{H}^{0}$, the class consisting of those functions $f$ in $\mathcal{S}_{H}^{*}$, $\mathcal{K}_{H}$ and $\mathcal{C}_{H}$ respectively, for which $f_{\bar{z}}(0)=0$.

Recall that convexity and starlikeness are not hereditary properties for univalent harmonic mappings. In \cite{sumit2}, the authors introduced the notion of fully starlike mappings of order $\beta$ and fully convex mappings of order $\beta$ ($0\leq \beta <1$) that are characterized by the conditions
\[\frac{\partial}{\partial \theta}\arg f(r e^{i \theta})> \beta\quad \mbox{and}\quad\frac{\partial}{\partial \theta}\left(\arg \left\{\frac{\partial}{\partial \theta}f(r e^{i \theta})\right\}\right)> \beta,\quad (0\leq\theta< 2\pi, 0<r<1)\]
respectively. For $\beta=0$ these classes were studied by Chuaqui, Duren and Osgood \cite{chuaqui}. Let $\mathcal{FS}^{*}_{H}(\beta)$ and $\mathcal{FK}_{H}(\beta)$ $(0\leq\beta<1)$ denote the subclasses of $\mathcal{S}^{*}_{H}$ and $\mathcal{K}_{H}$ respectively consisting of fully starlike functions of order $\beta$ and fully convex functions of order $\beta$. Set $\mathcal{FS}_{H}^{*0}(\beta)=\mathcal{FS}_{H}^{*}(\beta) \cap \mathcal{S}_{H}^{*0}$ and $\mathcal{FK}_{H}^{0}(\beta)=\mathcal{FK}_{H}(\beta) \cap \mathcal{K}_{H}^{0}$.

Clunie and Sheil-Small \cite{cluniesheilsmall} gave a sufficient condition for a harmonic function to be univalent close-to-convex.

\begin{lemma}\cite[Lemma 5.15, p.\ 19]{cluniesheilsmall}\label{lem1.1}
Suppose that $H$, $G$ are analytic in $\mathbb{D}$ with $|G'(0)|<|H'(0)|$ and that $H+\epsilon G$ is close-to-convex for each $|\epsilon|=1$. Then $F=H+\overline{G}$ is harmonic univalent and close-to-convex in $\mathbb{D}$.
\end{lemma}

Making use of Lemma \ref{lem1.1}, Clunie and Sheil-Small \cite{cluniesheilsmall} proved that if $f=h+\overline{g}$ is sense-preserving in $\mathbb{D}$ and $h+\epsilon g$ is convex for some $\epsilon$ $(|\epsilon|\leq 1)$, then $f$ is harmonic univalent and close-to-convex in $\mathbb{D}$. A particular case of this result is the following.

\begin{lemma}\label{lem1.2}
Let $f=h+\overline{g} \in \mathcal{H}$ be sense-preserving and $h \in \mathcal{K}$. Then $f \in \mathcal{C}_{H}^{0}$.
\end{lemma}

The conditions in the hypothesis of Lemma \ref{lem1.2} can't be relaxed, that is, if $f=h+\overline{g} \in \mathcal{H}$ is sense-preserving and $h$ is non-convex, then $f$ need not be even univalent. Similarly the conclusion of Lemma \ref{lem1.2} can't be strengthened, that is, if $f=h+\overline{g} \in \mathcal{H}$ is sense-preserving and $h \in \mathcal{K}$, then $f$ need not map $\mathbb{D}$ onto a starlike or convex domain. These two statements are illustrated by examples in Section 2 of the paper. In addition, we will consider the cases under which a sense-preserving harmonic function $f=h+\overline{g} \in \mathcal{H}$ with $h \in \mathcal{K}$ belongs to $\mathcal{FS}_H^{*0}(\beta)$ or $\mathcal{FK}_H^{0}(\beta)$. The following lemma will be needed in our investigation.

\begin{lemma}\cite{sumit1}\label{lem1.3}
Let $f=h+\bar{g} \in \mathcal{H}$ where $h$ and $g$ are given by \eqref{eq1.1} with $b_1=g'(0)=0$. Suppose that $\lambda \in (0,1]$.
\begin{itemize}
\item[(i)] If $\sum_{n=2}^{\infty}n(|a_n|+|b_n|)\leq \lambda$ then $f$ is fully starlike of order $2(1-\lambda)/(2+\lambda)$.
\item[(ii)] If $\sum_{n=2}^{\infty}n^2(|a_n|+|b_n|)\leq \lambda$ then $f$ is fully starlike of order $2(2-\lambda)/(4+\lambda)$. Moreover, $f$ is fully convex of order $2(1-\lambda)/(2+\lambda)$.
\end{itemize}
All these results are sharp.
\end{lemma}

For $\alpha \in \mathbb{C}$ with $|\alpha|\leq 1$, let $\mathcal{M}(\alpha)$ denote the set of all harmonic functions $f=h+\bar{g}\in \mathcal{H}$ that satisfy
\[g'(z)=\alpha zh'(z) \quad \mbox{and}\quad \RE \left(1+\frac{zh''(z)}{h'(z)}\right)>-\frac{1}{2} \mbox{ for all } z \in \mathbb{D}.\]
In \cite{mocanu}, Mocanu conjectured that the functions in the class $\mathcal{M}(1)$ are univalent in $\mathbb{D}$. In \cite{lyzzaik}, Bshouty and Lyzzaik proved this conjecture by establishing that $\mathcal{M}(1)\subset \mathcal{C}_{H}^{0}$. In fact, $\mathcal{M}(\alpha)\subset \mathcal{C}_{H}^{0}$ for each $|\alpha|\leq 1$. The coefficient estimates, growth results, area theorem and convolution properties for the class $\mathcal{M}(\alpha)$ $(|\alpha|\leq 1)$ are obtained in the last section of the paper. The bounds for the radius of starlikeness and convexity of the class $\mathcal{M}(\alpha)$ are also determined. The bound for the radius of convexity turns out to be sharp for the class $\mathcal{M}(1)$ with the extremal function
\[F(z):=\RE \frac{z}{(1-z)^2}+i \IM \frac{z}{1-z} \in \mathcal{M}(1).\]
The radius of starlikeness of $F$ is also determined. The convolution properties of $F$ with certain right-half plane mappings are also discussed.

\section{Sufficient conditions for starlikeness and convexity}
Neither the conditions in the hypothesis of Lemma \ref{lem1.2} can be relaxed nor the conclusion of Lemma \ref{lem1.2} can be strengthened. The first two examples of this section verify the truth of this statement.

\begin{example}
Let $h(z)=z-z^2/2 \in \mathcal{S}^*$ and $g(z)=z^2/2-z^3/3$ so that $h'(z)=zg'(z)$. Then $h$ is non-convex and $f=h+\overline{g}$ is sense-preserving in $\mathbb{D}$. But $f$ is not even univalent in $\mathbb{D}$ since $f(z_0)=f(\overline{z}_0)=3/4$ where $z_0=(3+\sqrt{3} i)/4 \in \mathbb{D}$.
\end{example}

\begin{example}
Let $h(z)=z/(1-z) \in \mathcal{K}$ and $g'(z)=zh'(z)$. Then the function
\[f(z)=h(z)+\overline{g(z)}=\frac{z}{1-z}+\overline{\frac{z}{1-z}+\log (1-z)}\]
belongs to $\mathcal{C}_{H}^{0}$ by Lemma \ref{lem1.2}. The image of the unit disk under $f$ is shown in Figure \ref{fig2.1} as plots of the images of equally spaced radial segments and concentric circles. Clearly $f(\mathbb{D})$ is a non-starlike domain.

\begin{figure}[here]
\centering
\includegraphics{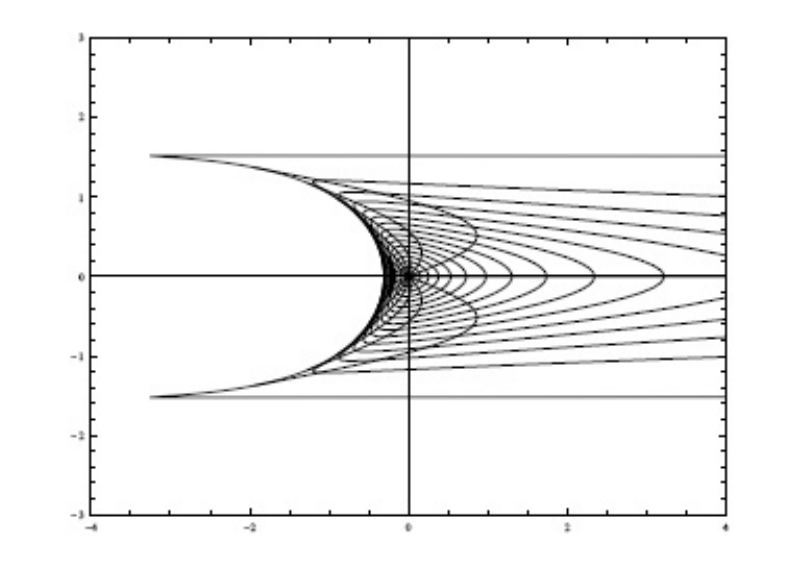}
\caption{Graph of the function $f(z)=2\RE \frac{z}{1-z}+\log (1-\overline{z})$.}\label{fig2.1}
\end{figure}
\end{example}

Now we will consider certain cases under which the conclusion of Lemma \ref{lem1.2} can be extended to fully starlike mappings of order $\beta$ and fully convex mappings of order $\beta$ $(0\leq \beta <1)$.

\begin{theorem}\label{th2.3}
Let $f=h+\overline{g} \in \mathcal{H}$ where $h$ and $g$ are given by \eqref{eq1.1}, and let $\alpha \in \mathbb{C}$. Further, assume that
\[g'(z)=\alpha z h'(z)\quad(z \in \mathbb{D})\quad \mbox{and}\quad \sum_{n=2}^{\infty} n^2 |a_n|\leq 1.\]
If $|\alpha|\leq 1$ then $f$ is univalent close-to-convex. If $|\alpha |\leq 1/3$ then $f$ is fully starlike of order $2(1-3|\alpha|)/(5+3|\alpha|)$.
\end{theorem}

\begin{proof}
The coefficient inequality $\sum_{n=2}^{\infty} n^2 |a_n|\leq 1$ implies that $h \in \mathcal{K}$ (see \cite{alexander}). So $f \in \mathcal{C}_{H}^{0}$ by Lemma \ref{lem1.2} if $|\alpha|\leq 1$. The relation $g'(z)=\alpha z h'(z)$ gives $b_1=0$ and
\[b_n=\alpha \frac{n-1}{n}a_{n-1},\quad (n\geq 2;\, a_1=1).\]
Consider
\begin{align*}
\sum_{n=2}^{\infty} n(|a_n|+|b_n|)&=\sum_{n=2}^{\infty} (n|a_n|+(n-1)|\alpha||a_{n-1}|)\\
                                  &=(1+|\alpha|)\sum_{n=2}^{\infty} n|a_n|+|\alpha|\\
                                  &\leq\frac{1}{2}(1+|\alpha|)\sum_{n=2}^{\infty}n^2 |a_n|+|\alpha|\\
                                  &\leq \frac{1}{2}(1+|\alpha|)+|\alpha|=\frac{1+3|\alpha|}{2}.
\end{align*}
By applying Lemma \ref{lem1.3}(i), it follows that $f \in \mathcal{FS}_{H}^{*0}(2(1-3|\alpha|)/(5+3|\alpha|))$ if $|\alpha|\leq 1/3$.
\end{proof}

\begin{theorem}
Let $f=h+\overline{g} \in \mathcal{H}$ where $h$ and $g$ are given by \eqref{eq1.1}, and let $\alpha \in \mathbb{C}$. Further, assume that
\[g'(z)=\alpha z h'(z)\quad(z \in \mathbb{D})\quad \mbox{and}\quad \sum_{n=2}^{\infty} n^3 |a_n|\leq 1.\]
If $|\alpha |\leq 2/11$ then $f$ is fully starlike of order $2(6-11|\alpha|)/(18+11|\alpha|)$. Moreover, $f$ is fully convex of order $2(2-11|\alpha|)/(10+11|\alpha|)$.
\end{theorem}

\begin{proof}
To apply Lemma \ref{lem1.3}(ii), consider the sum
\begin{align*}
\sum_{n=2}^{\infty} n^2(|a_n|+|b_n|)&=\sum_{n=2}^{\infty} (n^2|a_n|+n(n-1)|\alpha||a_{n-1}|)\\
                                  &=(1+|\alpha|)\sum_{n=2}^{\infty} n^2|a_n|+2|\alpha|+|\alpha|\sum_{n=2}^{\infty} n|a_n|\\
                                  &\leq\frac{1}{2}(1+|\alpha|)\sum_{n=2}^{\infty}n^3 |a_n|+2|\alpha|+\frac{1}{4}|\alpha|\sum_{n=2}^{\infty} n^3|a_n|\\
                                  &\leq \frac{1}{2}(1+|\alpha|)+2|\alpha|+\frac{1}{4}|\alpha|=\frac{2+11|\alpha|}{4}.
\end{align*}
Hence $f \in \mathcal{FS}_{H}^{*0}(2(6-11|\alpha|)/(18+11|\alpha|)) \cap \mathcal{FK}_{H}^0(2(2-11|\alpha|)/(10+11|\alpha|))$.
\end{proof}

\begin{remark}
If $h \in \mathcal{K}$ then the harmonic function $f=h+\epsilon \overline{h}$ is univalent and fully convex of order $0$ for each $|\epsilon|<1$. Similarly, if $h \in \mathcal{S}^*$ then the harmonic function $f=h+\epsilon \overline{h}$ is fully starlike of order $0$ for each $|\epsilon|<1$.
\end{remark}

\section{Class $\mathcal{M}(\alpha)$ ($|\alpha|\leq 1$)}\label{sec2}
In this section, we will investigate the properties of functions in the class $\mathcal{M}(\alpha)$.

\begin{theorem}\label{th3.1}
Let $\alpha \in \mathbb{C}$ with $|\alpha|\leq 1$. Then we have the following.
\begin{itemize}
  \item [(i)] $\mathcal{M}(\alpha)\subset \mathcal{C}_{H}^{0}$.
  \item [(ii)] (Coefficient estimates) If $f=h+\bar{g} \in \mathcal{M}(\alpha)$ where $h$ and $g$ are given by \eqref{eq1.1}, then $b_1=g'(0)=0$,
\[|a_n|\leq \frac{n+1}{2}\quad \mbox{and}\quad|b_n|\leq\frac{n-1}{2}|\alpha|\]
for $n=2,3,\ldots$. Moreover, these bounds are sharp for each $\alpha$.
  \item [(iii)] (Growth theorem) The inequality
\[|f(z)|\leq \frac{|z|}{(1-|z|)^2}\left[1-\frac{1}{2}(1-|\alpha|)|z|\right],\quad z\in \mathbb{D},\]
holds for every function $f \in \mathcal{M}(\alpha)$. This bound is sharp.
  \item [(iv)] (Area theorem) The area  of the image of each function $f \in \mathcal{M}(\alpha)$ is greater than or equal to $\pi(1-|\alpha|^{2}/2)$ and this minimum is attained only for the function $g_{\alpha}(z)=z+\alpha \bar{z}^2/2$.
\end{itemize}
\end{theorem}

\begin{proof}
To prove (i), let $f=h+\bar{g} \in \mathcal{M}(\alpha)$. Since $|g'(0)|<|h'(0)|$, it suffices to show that the analytic functions $F_{\epsilon}=h+\epsilon g$ are close-to-convex in $\mathbb{D}$ for each $|\epsilon|=1$, in view of Lemma \ref{lem1.1}. It is easy to verify that
\begin{align*}
\RE \left(1+\frac{zF''_{\epsilon}(z)}{F'_{\epsilon}(z)}\right)&=\RE \frac{\alpha \epsilon z}{1+\alpha \epsilon z}+\RE\left(1+\frac{zh''(z)}{h'(z)}\right)\\
                                                              &=\frac{1}{2}-\frac{1}{2}\frac{1-|\alpha \epsilon z|^2}{|1+\alpha \epsilon z|^2}+\RE\left(1+\frac{zh''(z)}{h'(z)}\right)\\
                                                              &=\frac{1}{2}-\frac{1}{2}P_{\zeta}(\theta)+\RE\left(1+\frac{zh''(z)}{h'(z)}\right)>-\frac{1}{2}P_{\zeta}(\theta)
\end{align*}
where $z=r e^{i \theta}$ ($0<r<1$), $\zeta=-\bar{\alpha}\bar{\epsilon} r$ and $P_{\zeta}(\theta)=(1-|\zeta|^2)/|e^{i \theta}-\zeta|^2$ $(|\zeta|<1)$ is the Poisson Kernel. Fix $\theta_1$, $\theta_2$ with $0<\theta_2 -\theta_1<2\pi$. Then
\[\int_{\theta_1}^{\theta_2} \RE \left(1+\frac{r e^{i \theta}F''_{\epsilon}(r e^{i \theta})}{F'_{\epsilon}(r e^{i \theta})}\right)\,d\theta \geq -\frac{1}{2}\int_{\theta_1}^{\theta_2}P_{\zeta}(\theta)\,d\theta=-\pi.\]
By the well-known Kaplan's theorem \cite{kaplan}, it follows that each $F_{\epsilon}$ is close-to-convex in $\mathbb{D}$ and hence $f \in \mathcal{C}_{H}^{0}$. This proves (i).

For the proof of (ii), note that the function $h$ is close-to-convex of order $1/2$ and hence its coefficients satisfy $|a_n|\leq (n+1)/2$ for $n=2,3,\ldots$ (see \cite{goodman,owa}). Regarding the bound for $b_n$, note that the relation $g'(z)=\alpha zh'(z)$ gives
\[(n+1)b_{n+1}=n\alpha a_{n},\quad n=1,2,\ldots\]
so that $|b_n|\leq (n-1)|\alpha|/2$. For sharpness, consider the functions
\begin{equation}\label{eq3.1}
f_{\alpha}(z)=\frac{1}{2}\left(\frac{z}{1-z}+\frac{z}{(1-z)^2}\right)-\overline{\frac{1}{2}\alpha\left(\frac{z}{1-z}-\frac{z}{(1-z)^2}\right)}\quad z \in \mathbb{D}, \quad |\alpha|\leq 1.
\end{equation}
The functions $f_\alpha \in \mathcal{M}(\alpha)$ for each $|\alpha|\leq 1$ and
\[f_{\alpha}(z)=z+\sum_{n=2}^{\infty}\frac{n+1}{2}z^{n}+\overline{\sum_{n=2}^{\infty}\frac{n-1}{2}\alpha z^{n}},\]
showing that the bounds are best possible. Figure \ref{fig3.1} depicts the image domain $f_{\alpha}(\mathbb{D})$ for $\alpha=\pm 1, \pm i$.

Using the estimates for $|a_n|$ and $|b_n|$, it follows that
\begin{align*}
|f(z)|&\leq |z|+\sum_{n=2}^{\infty}|a_{n}||z|^{n}+\sum_{n=2}^{\infty}|b_{n}||z|^{n}\\
      &\leq |z|+\frac{1}{2}\sum_{n=2}^{\infty}(n+1)|z|^{n}+\frac{1}{2}|\alpha|\sum_{n=2}^{\infty}(n-1)|z|^{n}\\
      &=|z|+\frac{1}{2}(1+|\alpha|)\sum_{n=2}^{\infty}n|z|^{n}+\frac{1}{2}(1-|\alpha|)\sum_{n=2}^{\infty}|z|^n\\
      &=\frac{|z|}{(1-|z|)^2}\left[1-\frac{1}{2}(1-|\alpha|)|z|\right].
\end{align*}
The bound is sharp with equality holding for the function $f_\alpha$ given by \eqref{eq3.1}. This proves (iii).

\begin{figure}[here]
  \centering
  \includegraphics{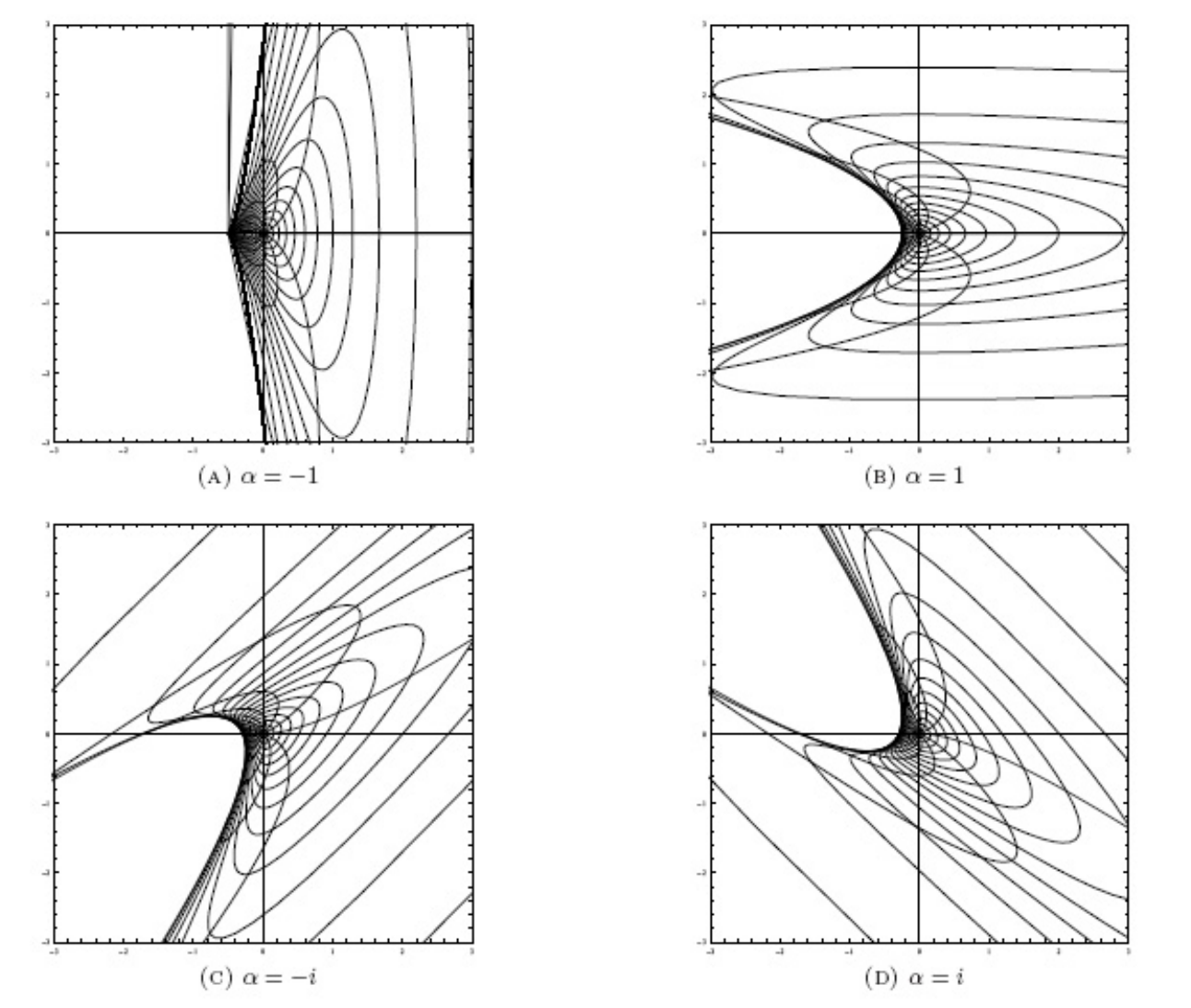}
 
  \caption{Graph of the function $f_{\alpha}$ for different values of $\alpha$.}\label{fig3.1}
\end{figure}

For the last part of the theorem, suppose that $f=h+\bar{g} \in \mathcal{M}(\alpha)$, where $h$ and $g$ are given by \eqref{eq1.1}. Then the area of the image $f(\mathbb{D})$ is
\begin{equation*}
\begin{split}
A&=\int \!\!\! \int_{\mathbb{D}}(|h'(z)|^{2}-|g'(z)|^{2})\,dx\,dy\\
 &=\int \!\!\! \int_{\mathbb{D}}|h'(z)|^2\,dx\,dy-|\alpha|^{2}\int \!\!\! \int_{\mathbb{D}}|zh'(z)|^2\,dx\,dy\\
 &=\pi\left(1+\sum_{n=2}^{\infty}n|a_n|^2\right)-|\alpha|^2\left(\frac{\pi}{2}+\pi \sum_{n=2}^{\infty}\frac{n^2}{n+1}|a_n|^2\right),\\
 &=\pi\left(1-\frac{1}{2}|\alpha|^2\right)+\pi\sum_{n=2}^{\infty}\left(n-\frac{n^2}{n+1}|\alpha|^2\right)|a_n|^2.
\end{split}
\end{equation*}
The last sum is minimized by choosing $a_n=0$ for $n=2,3,\ldots$. This gives $h(z)=z$ so that $g(z)=\alpha z^2/2$. This completes the proof of the theorem.
\end{proof}

If $\alpha=0$, then the family $\mathcal{M}(\alpha)$ reduces to the class of normalized analytic functions $f$ with $f(0)=0=f'(0)-1$ satisfying $\RE (1+zf''(z)/f'(z))>-1/2$ for all $z \in \mathbb{D}$. Ozaki \cite{ozaki} independently proved that the functions in the class $\mathcal{M}(0)$ are univalent in $\mathbb{D}$.

For analytic functions $f(z)=z+\sum_{n=2}^{\infty}a_{n}z^{n}$ and $F(z)=z+\sum_{n=2}^{\infty}A_{n}z^{n}$, their convolution (or Hadamard product) is defined as $(f*F)(z)=z+\sum_{n=2}^{\infty}a_{n}A_{n}z^{n}$. In the harmonic case, with
\begin{align*}
f&=h+\bar{g}=z+\sum_{n=2}^{\infty}a_{n}z^{n}+\overline{\sum_{n=1}^{\infty}b_{n}z^{n}},\quad \mbox{and}\\
F&=H+\bar{G}=z+\sum_{n=2}^{\infty}A_{n}z^{n}+\overline{\sum_{n=1}^{\infty}B_{n}z^{n}}.
\end{align*}
their harmonic convolution is defined as
\[f*F=h*H+\overline{g*G}=z+\sum_{n=2}^{\infty}a_{n}A_{n}z^{n}+\overline{\sum_{n=1}^{\infty}b_{n}B_{n}z^{n}}.\]
Results regarding harmonic convolution can be found in \cite{cluniesheilsmall,dorff2,sumit1,sumit2}.

\begin{remark}
Fix $\alpha$ with $|\alpha|\leq 1$. It is easy to see that the Hadamard product of two functions in $\mathcal{M}(\alpha)$ need not necessarily belong to $\mathcal{M}(\alpha)$. For instance, consider the function $f_{\alpha}$ given by \eqref{eq3.1}. The coefficients of the product $f_{\alpha}*f_{\alpha}$ are too large for this product to be in $\mathcal{M}(\alpha)$ in view of Theorem \ref{th3.1}(ii).
\end{remark}

In \cite{cluniesheilsmall}, Clunie and Sheil-Small showed that if $\varphi \in \mathcal{K}$ and $f \in \mathcal{K}_{H}$ then the functions $(\beta \overline{\varphi}+\varphi)*f \in \mathcal{C}_{H}$ ($|\beta|\leq 1$). The result is even true if $\mathcal{K}_{H}$ is replaced by $\mathcal{M}(\alpha)$.

\begin{theorem}
Let $\varphi \in \mathcal{K}$ and $f \in \mathcal{M}(\alpha)$ $(|\alpha|\leq 1)$. Then the functions $(\beta \overline{\varphi}+\varphi)*f \in \mathcal{C}_{H}^{0}$ for $|\beta|\leq 1$.
\end{theorem}

\begin{proof}
Writing $f=h+\bar{g}$ we have
\[(\beta \overline{\varphi}+\varphi)*f=\varphi*h+\overline{\bar{\beta}(\varphi*g)}=H+\overline{G},\]
where $H=\varphi*h$ and $G=\bar{\beta}(\varphi*g)$ are analytic in $\mathbb{D}$ with $|G'(0)|<|H'(0)|$. Setting $F=H+\epsilon G=\varphi*(h+ \bar{\beta}\epsilon g)$ where $|\epsilon|=1$, we note that $F$ is close-to-convex in $\mathbb{D}$ since $h+ \bar{\beta}\epsilon g \in \mathcal{C}$, $\varphi \in \mathcal{K}$ and $\mathcal{K}*\mathcal{C}\subset \mathcal{C}$. By Lemma \ref{lem1.1}, it follows that $H+\overline{G}$ is harmonic close-to-convex, as desired.
\end{proof}

\begin{remark}\label{rem3.4}
The function $f_{-1} \in \mathcal{M}(-1)$ given by \eqref{eq3.1} is the harmonic half-plane mapping
\begin{equation}\label{eq3.2}
L(z):=f_{-1}(z)=\RE \left(\frac{z}{1-z}\right)+i \IM\left(\frac{z}{(1-z)^2}\right)
\end{equation}
constructed by shearing the conformal mapping $l(z)=z/(1-z)$ vertically with dilatation $w(z)=-z$ (see Figure \ref{fig3.1}(A)). Note that
\[(L*L)(z)=z+\sum_{n=2}^{\infty}\left(\frac{n+1}{2}\right)^2z^{n}+\overline{\sum_{n=2}^{\infty}\left(\frac{n-1}{2}\right)^2z^{n}},\quad z \in \mathbb{D}\]
is univalent in $\mathbb{D}$ by \cite[Theorem 3]{dorff2}. In fact, the image of the unit disk $\mathbb{D}$ under $L*L$ is $\mathbb{C}\backslash(-\infty,-1/4]$ which shows that $L*L \in \mathcal{S}_{H}^{*0}$. Since $f_{\alpha}*f_{\bar{\alpha}}=L*L$ for each $|\alpha|=1$, where the functions $f_{\alpha}$ are given by \eqref{eq3.1}, it follows that the Hadamard product $f_{\alpha}*f_{\bar{\alpha}}$ is univalent and starlike in $\mathbb{D}$ for each $|\alpha|=1$.
\end{remark}

The next theorem determines the bounds for the radius of starlikeness and convexity of the class $\mathcal{M}(\alpha)$.

\begin{theorem}\label{th3.5}
Let $\alpha \in \mathbb{C}$ with $|\alpha|\leq 1$.
\begin{itemize}
\item [(a)] Each function in $\mathcal{M}(\alpha)$ maps the disk $|z|<2-\sqrt{3}$ onto a convex domain.
\item [(b)] Each function in $\mathcal{M}(\alpha)$ maps the disk $|z|<4\sqrt{2}-5$ onto a starlike domain.
\end{itemize}
\end{theorem}

\begin{proof}
Let $f=h+\bar{g} \in \mathcal{M}(\alpha)$. Then the analytic functions $F_{\lambda}=h+\lambda g$ are close-to-convex in $\mathbb{D}$ for each $|\lambda|=1$ (see the proof of Theorem \ref{th3.1}(i)).

Since the radius of convexity in close-to-convex analytic mappings is $2-\sqrt{3}$, the functions $F_{\lambda}$ are convex in $|z|<2-\sqrt{3}$. In view of \cite[Theorem 2.3, p.\ 89]{sumit2}, it follows that $f$ is fully convex (of order 0) in $|z|<2-\sqrt{3}$. This proves (a).

Similarly, since the radius of starlikeness for close-to-convex analytic mappings is $4\sqrt{2}-5$, it follows that each $F_{\lambda}$ is starlike in $|z|<4\sqrt{2}-5$. By \cite[Theorem 2.7, p.\ 91]{sumit2}, $f$ is fully starlike (of order 0) in $|z|<4\sqrt{2}-5\approx 0.65685$.
\end{proof}

Now, we shall show that the bound $2-\sqrt{3}$ for the radius of convexity is sharp for the class $\mathcal{M}(1)$. To see this, consider the function $f_{1}$ given by \eqref{eq3.1}, which may be rewritten as
 \begin{equation}\label{eq3.3}
F(z):=f_1(z)=\RE \left(\frac{z}{(1-z)^2}\right)+i \IM \left(\frac{z}{1-z}\right).
\end{equation}
Its worth to note that the function $F$ may be constructed by shearing the conformal mapping $l(z)=z/(1-z)$ horizontally with dilatation $w(z)=z$. In \cite{hayami}, it has been shown that $F(\mathbb{D})=\{u+iv:v^2>-(u+1/4)\}$ (see Figure \ref{fig3.1}(B)). In particular, $F \not\in \mathcal{S}_{H}^{*0}$. For instance, $z_0=-1-i \in F(\mathbb{D})$ but $z_{0}/2 \not\in F(\mathbb{D})$. In fact, $z_{0}/2 \in \partial F(\mathbb{D})$. Thus $\mathcal{M}(1) \not\subset \mathcal{S}_{H}^{*0}$.

The next example determines the radius of convexity of the mapping $F$ by employing a calculation similar to the one carried out in \cite[Section 3.5]{duren}.

\begin{example}\label{ex3.6}
For the purpose of computing the radius of convexity of $F$, it is enough to study the change of the tangent direction
\[\Psi_{r}(\theta)=\arg\left\{\frac{\partial}{\partial\theta}F(re^{i\theta})\right\}\]
of the image curve as the point $z= r e^{i\theta}$ moves around the circle $|z|=r$. Note that
\[\frac{\partial}{\partial\theta}F(re^{i\theta})=A(r,\theta)+iB(r,\theta),\]
where
\[|1-z|^{6}A(r,\theta)=-r[(1-6r^2+r^4)\sin\theta +r(1+r^2) \sin 2\theta]\]
and
\[|1-z|^{4}B(r,\theta)=r[(1-r^2)\cos\theta-2r].\]

\begin{figure}[here]
\centering
\includegraphics{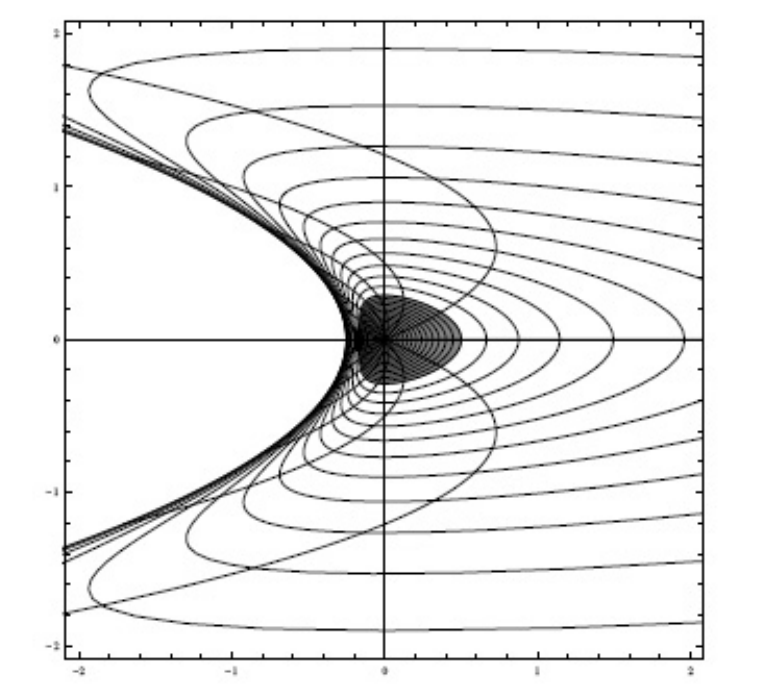}
\caption{$2-\sqrt{3}$ - the radius of convexity of $F$.}\label{fig3}
\end{figure}

The problem now reduces to finding the values of $r$ such that the argument of the tangent vector, or equivalently
\[\tan \Psi_{r}(\theta)=\frac{B(r,\theta)}{A(r,\theta)}=
\frac{(1-2r\cos{\theta}+r^{2})[2r-(1-r^{2})\cos{\theta}]}{(1-6r^2+r^4)\sin\theta +r(1+r^2) \sin 2\theta}\]
is a non-decreasing function of $\theta$ for $0<\theta<\pi$. A lengthy calculation leads to an expression for the derivative in the form
\[[(1-6r^2+r^4)+2r(1+r^2)u]^2(1-u^2)\frac{\partial}{\partial\theta}\tan \Psi_{r}(\theta)=p(r,u),\]
where $u=\cos \theta$ and
\begin{align*}
p(r,u)=1 + 4 r^2 - 26 r^4 + 4 r^6 &+ r^8 -6 u r (1 + r^2) (1 + r^4 - 6 r^2) \\
&- 12 r^2 u^2 (1 + r^2)^2 + 4 r u^3 (1 + r^2) (1 + r^4).
\end{align*}

Observe that the roots of $p(r,u)=0$ in $(0,1)$ are increasing as a function of $u \in [-1,1]$. Consequently, it follows that $p(r,u)\geq 0$ for $-1\leq u\leq 1$ if and only if
\[p(r,-1)=(1+r)^6(1-4r+r^2)\geq 0.\]
This inequality implies that $r\leq 2-\sqrt{3}$. This proves that the tangent angle $\Psi_{r}(\theta)$ increases monotonically with $\theta$ if $r\leq 2-\sqrt{3}$ but is not monotonic for $2-\sqrt{3}<r<1$. Thus, the harmonic mapping $F$ given by \eqref{eq3.3} sends each disk $|z|<r\leq 2-\sqrt{3}$ to a convex domain, but the image is not convex when $2-\sqrt{3}<r<1$ (see Figure \ref{fig3}).
\end{example}

Combining Theorem \ref{th3.5} and Example \ref{ex3.6}, it immediately follows that

\begin{theorem}
The radius of convexity of the class $\mathcal{M}(1)$ is $2-\sqrt{3}$. Moreover, the bound $2-\sqrt{3}$ is sharp.
\end{theorem}

The next example determines the radius of starlikeness of the mapping $F$ given by \eqref{eq3.3}.

\begin{example}\label{ex3.8}
The harmonic mapping $F$ given by \eqref{eq3.1}  sends each disk $|z|<r\leq r_0$ to a starlike domain, but the image is not starlike when $r_{0}<r<1$, where $r_{0}$ is given by \begin{equation}\label{eq3.2}
r_{0}=\frac{1}{3}\sqrt{\frac{1}{3}(37-8\sqrt{10})}\approx 0.658331.
\end{equation}
In this case, one needs to study the change of the direction $\Phi_{r}(\theta)=\arg F(r e^{i\theta})$ of the image curve as the point $z=r e^{i\theta}$ moves around the circle $|z|=r$. A direct calculation gives
\[F(r e^{i\theta})=C(r,\theta)+iD(r,\theta),\]
where
\[|1-z|^4C(r,\theta)=r[ (1+r^2)\cos \theta-2r]\quad \mbox{and}\quad |1-z|^{2}D(r,\theta)=r\sin\theta.\]
For our assertion, it suffices to show that
\[\tan \Phi_{r}(\theta)=\frac{D(r,\theta)}{C(r,\theta)}=\frac{\sin\theta(1-2r \cos\theta+r^2)}{(1+r^2)\cos\theta-2r}\]
is a nondecreasing function of $\theta$. A straightforward calculation leads to an expression for the derivative in the form
\[[(1+r^2)u-2r]^2 \frac{\partial}{\partial \theta}\tan \Phi_{r}(\theta)=q(r,u),\]
where $u=\cos\theta$ and
\[q(r,u)=(1-r^2)^2-2ru(1+r^2)+8r^2 u^2-2r(1+r^2)u^3.\]
The problem is now to find the values of the parameter $r$ for which the polynomial $q(r,u)$ is non-negative in the whole interval $-1\leq u \leq 1$. Observe that
\[q(r,-1)=(1+r)^4>0\quad\mbox{and}\quad q(r,1)=(1-r)^4>0.\]
Also, differentiation gives
\[\frac{\partial}{\partial u}q(r,u)=-2r(1+r^2)+16r^2u-6r(1+r^2)u^2,\]
showing that $q(r,u)$ has a local minimum at $u=(4r-\sqrt{-3+10r^2-3r^4})/(3(1+r^2))$ and a local maximum at $u=(4r+\sqrt{-3+10r^2-3r^4})/(3(1+r^2))$. Thus $q(r,u)\geq 0$ for $-1\leq u\leq 1$ if and only if
\begin{align*}
q\left(r,\frac{4r-\sqrt{-3+10r^2-3r^4}}{3(1+r^2)}\right)&=\frac{1}{27(1+r^2)^2}[27-72 r^2+58 r^4-72 r^6+27 r^8\\
                                                                                        &+4r(3+10r^2+3r^4)\sqrt{-3+10r^2-3r^4}] \geq 0.
\end{align*}
This inequality implies that $r\leq r_{0}$, where $r_0$ is given by \eqref{eq3.1}. This proves that the angle $\Phi_{r}(\theta)$ increases monotonically with $\theta$ if $r\leq r_{0}$ and hence the harmonic mapping $F$ sends each disk $|z|<r \leq r_{0}$ to a starlike domain, but the image is not starlike when $r_{0}<r<1$ (see Figure \ref{fig4}).

\begin{figure}[here]
\centering
\includegraphics{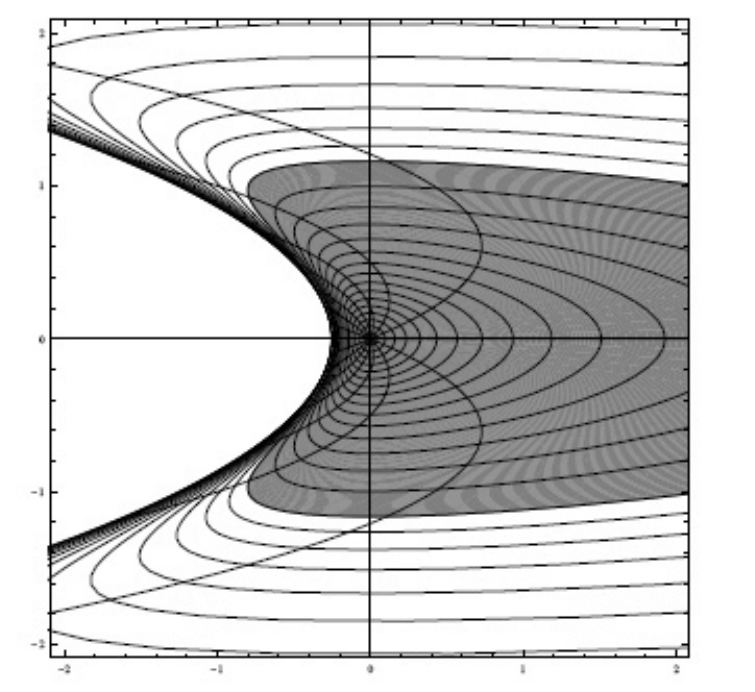}
\caption{$\frac{1}{3}\sqrt{\frac{1}{3}(37-8\sqrt{10})}$ - the radius of starlikeness of $F$.}\label{fig4}
\end{figure}

\end{example}

Combining Theorem \ref{th3.5} and Example \ref{ex3.8}, we have

\begin{theorem}
If $r_{S}$ is the radius of starlikeness of $\mathcal{M}(1)$, then
\[4\sqrt{2}-5\leq r_{S}\leq \frac{1}{3}\sqrt{\frac{1}{3}(37-8\sqrt{10})}.\]
\end{theorem}

By Remark \ref{rem3.4}, $F*F$ is univalent and starlike in $\mathbb{D}$. However, the product $L*F$ where $L$ is the harmonic half-plane mapping given by \eqref{eq3.2} is not even univalent, although it is sense-preserving in $\mathbb{D}$. In fact, the convolution of $F$ with certain right-half plane mappings is sense-preserving in $\mathbb{D}$. This is seen by the following theorem.

\begin{theorem}\label{th3.6}
Let $f=h+\bar{g} \in \mathcal{K}_{H}^{0}$ with $h(z)+g(z)=z/(1-z)$ and $w(z)=g'(z)/h'(z)=e^{i\theta} z^{n}$, where $\theta \in \mathbb{R}$. If $n=1,2$ then $F*f$ is locally univalent in $\mathbb{D}$, $F$ being given by \eqref{eq3.1}.
\end{theorem}

\begin{proof}
We need to show that the dilatation $\widetilde{w}$ of $F*f$ satisfies $|\widetilde{w}(z)|<1$ for all $z \in \mathbb{D}$. It is an easy exercise to derive the expression of dilatation $\widetilde{w}$ in the form
\[\widetilde{w}(z)=z \frac{w^2(z)+[w(z)-\frac{1}{2}w'(z)z]+\frac{1}{2}w'(z)}{1+[w(z)-\frac{1}{2}w'(z)z]+\frac{1}{2}w'(z)z^2}, \quad z \in \mathbb{D}.\]
The rest of the proof is similar to \cite[Theorem 3]{dorff2}.
\end{proof}

\end{document}